\documentclass[reqno, 10pt]{amsart}
\usepackage{amssymb}
\usepackage[usenames, dvipsnames]{color}
\usepackage{esint}
\usepackage{hyperref}
\usepackage{verbatim}

\usepackage{cite}
\usepackage[nobysame]{amsrefs}
\def\MR#1{}

\theoremstyle{plain}
\newtheorem{theorem}{Theorem}[section]
\newtheorem{lemma}[theorem]{Lemma}
\newtheorem{corollary}[theorem]{Corollary}
\newtheorem{proposition}[theorem]{Proposition}

\theoremstyle{definition}

\theoremstyle{remark}
\newtheorem{remark}[theorem]{Remark}

\newcommand{\dv}{\operatorname{div}}

\numberwithin{equation}{section}

\newcommand{\bN}{\mathbb{N}}

\newcommand{\bR}{\mathbb{R}}

\newcommand{\bZ}{\mathbb{Z}}
\newcommand{\bS}{\mathbb{S}}

\makeatletter
\def\dashint{\operatorname%
{\,\,\text{\bf--}\kern-.98em\DOTSI\intop\ilimits@\!\!}}
\makeatother

\begin{document}
\title[Optimal gradient estimates for the insulated conductivity problem]{Optimal gradient estimates of solutions to the insulated conductivity problem in dimension greater than two}
\author[H. Dong]{Hongjie Dong}
\author[Y.Y. Li]{YanYan Li}
\author[Z. Yang]{Zhuolun Yang}
\address[H. Dong]{Division of Applied Mathematics, Brown University, 182 George Street, Providence, RI 02912, USA}
\email{Hongjie\_Dong@brown.edu}

\address[Y.Y. Li]{Department of
Mathematics, Rutgers University, 110 Frelinghuysen Rd, Piscataway,
NJ 08854, USA}
\email{yyli@math.rutgers.edu}

\address[Z. Yang]{Institute for Computational and Experimental Research in Mathematics, Brown University, 121 South Main Street, Providence, RI 02903, USA}
\email{zhuolun\_yang@brown.edu}

\thanks{H. Dong is partially supported by Simons Fellows Award 007638, NSF Grant DMS-2055244, and the Charles Simonyi Endowment at the Institute of Advanced Study.}
\thanks{Y.Y. Li is partially supported by NSF Grants DMS-1501004, DMS-2000261, and Simons Fellows Award 677077.}
\thanks{Z. Yang is partially supported by NSF Grants DMS-1501004, DMS-2000261, and the Simons Bridge Postdoctoral Fellowship at Brown University}

\begin{abstract}
We study the insulated conductivity problem with inclusions embedded in a bounded domain in $\bR^n$. The gradient of solutions may blow up as $\varepsilon$, the distance between inclusions, approaches to $0$. It was known that the optimal blow up rate in dimension $n = 2$ is of order $\varepsilon^{-1/2}$. It has recently been proved that in dimensions $n \ge 3$, an upper bound of the gradient is of order $\varepsilon^{-1/2 + \beta}$ for some $\beta > 0$. On the other hand, optimal values of $\beta$ have not been identified. In this paper, we prove that when the inclusions are balls, the optimal value of $\beta$ is  $[~-(n-1)+\sqrt{(n-1)^2+4(n-2)}~]/4 \in (0,1/2)$ in dimensions $n \ge 3$.
\end{abstract}
\maketitle

\section{Introduction and main results}

First we describe the nature of the domain. Let $\Omega \subset \bR^n$ be a bounded domain with $C^{2}$ boundary containing two $C^{2,\gamma} (0 < \gamma < 1)$ relatively strictly convex open sets $D_{1}$ and $D_{2}$ with dist$(D_1 \cup D_2, \partial \Omega) > c > 0$. Let
$$\varepsilon: = \mbox{dist}(D_1, D_2)$$
and $\widetilde{\Omega} := \Omega \setminus \overline{(D_1 \cup D_2)}$. The conductivity problem can be modeled by the following elliptic equation:
\begin{equation*}
\begin{cases}
\mathrm{div}\Big(a_{k}(x)\nabla{u}_{k}\Big)=0&\mbox{in}~\Omega,\\
u_{k}=\varphi(x)&\mbox{on}~\partial\Omega,
\end{cases}
\end{equation*}
where $\varphi\in{C}^{2}(\partial\Omega)$ is given, and
$$a_{k}(x)=
\begin{cases}
k\in(0,\infty)&\mbox{in}~D_{1}\cup{D}_{2},\\
1&\mbox{in}~\widetilde{\Omega}.
\end{cases}
$$
In the context of electric conduction, the elliptic coefficients $a_k$ refer to conductivities, and the solutions $u_k$ represent voltage potential. From an engineering point of view, it is significant to estimate the magnitude of the electric fields in the narrow region between the inclusions, which is given by $|\nabla u_k|$. This problem is analogous to the Lam\'e system studied by Babu\v{s}ka, Andersson, Smith, and Levin \cite{BASL}, where they analyzed numerically that, when the ellipticity constants are bounded away from $0$ and infinity, the gradient of solutions remain bounded independent of $\varepsilon$, the distance between inclusions. Later, Bonnetier and Vogelius \cite{BV} proved that when $\varepsilon = 0$, $|\nabla u_k|$ is bounded for a fixed $k$ and circular inclusions $D_1$ and $D_2$ in dimension $n = 2$. This result was extended by Li and Vogelius \cite{LV} to general second order elliptic equations of divergence form with piecewise H\"older coefficients and general shape of inclusions $D_1$ and $D_2$ in any dimension. Furthermore, they established a stronger piecewise $C^{1,\alpha}$ control of $u_k$, which is independent of $\varepsilon$. Li and Nirenberg \cite{LN} further extended this global Lipschitz and piecewise $C^{1,\alpha}$ result to general second order elliptic systems of divergence form, including the linear system of elasticity. Some higher order derivative estimates in dimension $n=2$ were obtained in \cites{DZ,DL,JiKang}.

When $k$ equals to $\infty$ (inclusions are perfect conductors) or $0$ (insulators), it was shown in \cites{Kel, BudCar, Mar} that the gradient of solutions generally becomes unbounded as $\varepsilon \to 0$. Ammari et al. in \cite{AKLLL} and \cite{AKL} considered the perfect and insulated conductivity problems with circular inclusions in $\bR^2$, and established optimal blow-up rates $\varepsilon^{-1/2}$ in both cases.
Yun extended in \cite{Y1} and \cite{Y2} the results above allowing $D_1$ and $D_2$ to be any bounded strictly convex smooth domains.

The above gradient estimates in dimension $n = 2$ were localized and extended to higher dimensions by Bao, Li, and Yin in \cite{BLY1} and \cite{BLY2}. For the perfect conductor case, they proved in \cite{BLY1} that
\begin{equation*}
\begin{cases}
\| \nabla u \|_{L^\infty(\widetilde{\Omega})} \le C\varepsilon^{-1/2} \|\varphi\|_{C^2(\partial \Omega)} &\mbox{when}~n=2,\\
\| \nabla u \|_{L^\infty(\widetilde{\Omega})} \le C|\varepsilon \ln \varepsilon|^{-1} \|\varphi\|_{C^2(\partial \Omega)} &\mbox{when}~n=3,\\
\| \nabla u \|_{L^\infty(\widetilde{\Omega})} \le C\varepsilon^{-1} \|\varphi\|_{C^2(\partial \Omega)} &\mbox{when}~n\ge 4.
\end{cases}
\end{equation*}
These bounds were shown to be optimal in the paper and they are independent of the shape of inclusions, as long as the inclusions are relatively strictly convex. Moreover, many works have been done in characterizing the singular behavior of $\nabla u$, which are significant in practical applications. For further works on the perfect conductivity problem and closely related ones, see e.g. \cites{ACKLY,BT1,BT2,DL,KLY1,KLY2,L,LLY,LWX,BLL,BLL2,DZ,KL,CY,ADY,Gor,LimYun} and the references therein.

For the insulated conductivity problem, it was proved in \cite{BLY2} that
\begin{equation}\label{insulated_upper_bound}
\| \nabla u \|_{L^\infty(\widetilde{\Omega})} \le C\varepsilon^{-1/2} \|\varphi\|_{C^2(\partial \Omega)}\quad \mbox{when}~n\ge 2.
\end{equation}
The upper bound is optimal for $n = 2$ as mentioned above.

Yun \cite{Y3} studied the following free space insulated conductivity problem in $\bR^3$: Let $H$ be a harmonic function in $\bR^3$, $D_1 = B_1\left(0,0,1+\frac{\varepsilon}{2} \right)$, and $D_2 = B_1\left(0,0,-1-\frac{\varepsilon}{2} \right)$,
\begin{equation*}
\begin{cases}
\Delta{u}=0&\mbox{in}~\bR^3\setminus\overline{(D_1 \cup D_2)},\\
\frac{\partial u}{\partial \nu} = 0 &\mbox{on}~\partial{D}_{i},~i=1,2,\\
u(x)-H(x) = O(|x|^{-2})&\mbox{as}~|x| \to \infty.
\end{cases}
\end{equation*}
He proved that for some positive constant $C$ independent of $\varepsilon$,
\begin{equation}
\label{Yun_result}
\max_{|x_3|\le \varepsilon/2}|\nabla u(0,0,x_3)| \le C \varepsilon^{\frac{\sqrt{2}-2}{2}}.
\end{equation}
He also showed that this upper bound of $|\nabla u|$ on the $\varepsilon$-segment connecting $D_1$ and $D_2$ is optimal for $H(x) \equiv x_1$. Although this result does not provide an upper bound of $|\nabla u|$ in the complement of the $\varepsilon$-segment, it has added support to a long time suspicion that the upper bound $\varepsilon^{-1/2}$ obtained for dimension $n=3$ in \cite{BLY2} is not optimal.

The upper bound \eqref{insulated_upper_bound} was recently improved by Li and Yang \cite{LY2} to
\begin{equation*}\label{insulated_upper_bound_2}
\| \nabla u \|_{L^\infty(\widetilde{\Omega})} \le C\varepsilon^{-1/2 + \beta} \|\varphi\|_{C^2(\partial \Omega)}\quad \mbox{when}~n\ge 3,
\end{equation*}
for some $\beta > 0$. When insulators are unit balls, a more explicit constant $\beta(n)$ was given by Weinkove in \cite{We} for $n \ge 4$ by a different method. The constant $\beta(n)$ obtained in \cite{We} presumably improves that in \cite{LY2}. In particular, it was proved in \cite{We} that $\beta(n)$ approaches $1/2$ from below as $n \to \infty$. However, the optimal blow up rate in dimensions $n \ge 3$ remained unknown. We draw reader's attention to a recent survey paper \cite{Kang} by Kang, where in the conclusions section, the three-dimensional case is described as an outstanding problem.

In this paper, we focus on the following insulated conductivity problem in dimensions $n \ge 3$, and give an optimal gradient estimate for a certain class of inclusions including two balls of any size:
\begin{equation}\label{equzero}
\left\{
\begin{aligned}
-\Delta u &=0 \quad \mbox{in }\widetilde{\Omega},\\
\frac{\partial u}{\partial \nu} &= 0 \quad \mbox{on}~\partial{D}_{i},~i=1,2,\\
 u &= \varphi \quad \mbox{on } \partial \Omega,
\end{aligned}
\right.
\end{equation}
where $\varphi\in{C}^{2}(\partial\Omega)$ is given, and $\nu = (\nu_1, \ldots, \nu_n)$ denotes the inner normal vector on $\partial{D}_{1} \cup \partial{D}_{2}$.

We use the notation $x = (x', x_n)$, where $x' \in \bR^{n-1}$. After choosing a coordinate system properly, we can assume that near the origin,  the part of $\partial D_1$ and $\partial D_2$, denoted by $\Gamma_+$ and $\Gamma_-$, are respectively the graphs of two $C^{2,\gamma}(0 < \gamma < 1)$ functions in terms of $x'$. That is, for some $R_0 > 0$,
\begin{align*}
\Gamma_+ =& \left\{ x_n = \frac{\varepsilon}{2}+f(x'),~|x'|<R_0\right\},\\
\Gamma_- =& \left\{ x_n = -\frac{\varepsilon}{2}+g(x'),~|x'|<R_0\right\},
\end{align*}
where $f$ and $g$ are $C^{2,\gamma}$ functions satisfying
$$f(x')>g(x')\quad\mbox{for}~~0<|x'|<R_{0},$$
\begin{equation}\label{fg_0}
f(0')=g(0')=0,\quad\nabla_{x'}f(0')=\nabla_{x'}g(0')=0,
\end{equation}
\begin{equation}\label{fg_1}
f(x')-g(x') =  a |x'|^2 + O(|x'|^{2+\gamma})\quad\mbox{for}~~0<|x'|<R_{0},
\end{equation}
with $a >0$. Here and throughout the paper, we use the notation $O(A)$ to denote a quantity that can be bounded by $CA$, where $C$ is some positive constant independent of $\varepsilon$. For $0 < r\leq R_{0}$, we denote
\begin{align}\label{domain_def_Omega}
\Omega_{r}:=\left\{(x',x_{n})\in \widetilde{\Omega}~\big|~-\frac{\varepsilon}{2}\right.&\left.+g(x')<x_{n}<\frac{\varepsilon}{2}+f(x'),~|x'|<r\right\}.
\end{align}
By standard elliptic estimates, the solution $u \in H^1(\widetilde\Omega)$ of \eqref{equzero} satisfies
\begin{equation}
\label{u_C1_outside}
\|u\|_{C^1(\widetilde\Omega \setminus \Omega_{R_0/2})} \le C.
\end{equation}
We will focus on the following problem near the origin:
\begin{equation}\label{main_problem_narrow}
\left\{
\begin{aligned}
-\Delta u &=0 \quad \mbox{in }\Omega_{R_0},\\
\frac{\partial u}{\partial \nu} &= 0 \quad \mbox{on } \Gamma_+ \cup \Gamma_-,\\
\|u\|&_{L^\infty(\Omega_{R_0})} \le 1.
\end{aligned}
\right.
\end{equation}
It was proved in \cite{BLY2} that for $u \in H^1(\Omega_{R_0})$ satisfying \eqref{main_problem_narrow},
\begin{equation}\label{grad_u_bound_rough}
| \nabla u(x)| \le C(\varepsilon + |x'|^2)^{-1/2}, \quad \forall x \in \Omega_{R_0},
\end{equation}
where $C$ is a positive constant depending only on $n, R_0, a, \|f\|_{C^2}$, and $\|g\|_{C^2}$, and is in particular independent of $\varepsilon$. The above mentioned improvement on \eqref{insulated_upper_bound} in \cites{LY2, We} also apply to \eqref{grad_u_bound_rough}.

Our main results of this paper are as follows.

\begin{theorem}
\label{radial_thm}
For $n \ge 3$, $\varepsilon \in (0,1/4)$, let $u \in H^1(\Omega_{R_0})$ be a solution of \eqref{main_problem_narrow} with $f,g$ satisfying \eqref{fg_0} and \eqref{fg_1}. Then there exists a positive constant $C$ depending only on $n$, $R_0$, $\gamma$, a positive lower bound of $a$, and an upper bound of $\|f\|_{C^{2,\gamma}}$ and $\|g\|_{C^{2,\gamma}}$, such that
\begin{equation}
\label{main_goal}
|\nabla u(x)|\le C \|u\|_{L^\infty(\Omega_{R_0})} (\varepsilon + |x'|^2)^{\frac{\alpha-1}{2}}\quad \forall x\in \Omega_{R_0/2},
\end{equation}
where
\begin{equation}
\label{alpha_n}
\alpha=\alpha(n):= \frac {-(n-1)+\sqrt{(n-1)^2+4(n-2)}}{2}\in (0,1).
\end{equation}
\end{theorem}

Note that $\alpha(n)$ is monotonically increasing in $n$, and
$$
\alpha(n) = 1 - \frac{2}{n} + O(\frac{1}{n^2}) \quad \mbox{as}~~n \to \infty.
$$
For $n = 3$, the exponent $\frac{\alpha - 1}{2} = \frac{\sqrt{2}-2}{2}$ is the same as the exponent in \eqref{Yun_result}. For $n \ge 4$, the exponent $\frac{\alpha-1}{2}$ is strictly greater than the one obtained in \cite{We}.

A consequence of Theorem \ref{radial_thm} is, in view of \eqref{u_C1_outside}, as follows.
\begin{corollary}
For $n \ge 3$, $\varepsilon \in (0,1/4)$, let $D_1, D_2$ be two balls of radius $r_1, r_2$, center at $(0', \varepsilon/2 + r_1)$ and $(0', - \varepsilon/2 -r_2)$, respectively. Let $u \in H^1(\widetilde\Omega)$ be the solution of $\eqref{equzero}$. Then there exists a positive constant $C$ depending only on $n$, $r_1$, $r_2$, and $\| \partial \Omega \|_{C^2}$ such that
\begin{equation}
\label{grad_u_upperbound_balls}
\|\nabla u\|_{L^\infty(\widetilde\Omega)} \le C \|\varphi\|_{C^2(\partial \Omega)} \varepsilon^{\frac{\alpha - 1}{2}},
\end{equation}
where $\alpha$ is given by \eqref{alpha_n}.
\end{corollary}

Estimate \eqref{grad_u_upperbound_balls} is optimal as shown in the following theorem.

\begin{theorem}
\label{optimality}
For $n \ge 3$, $\varepsilon \in (0,1/4)$, let $\Omega = B_5$, and $D_1 , D_2$ be the unit balls center at $(0', 1 + \varepsilon/2)$ and $(0', -1 - \varepsilon/2)$, respectively. Let $\varphi = x_1$ and $u \in H^1(\widetilde\Omega)$ be the solution of \eqref{equzero}. Then there exists positive constant $C$ depending only on $n$, such that
\begin{equation}
\label{grad_u_lower_bound}
\|\nabla u\|_{L^\infty(\widetilde\Omega \cap B_{2\sqrt{\varepsilon}})} \ge \frac{1}{C}\varepsilon^{\frac{\alpha-1}{2}},
\end{equation}
where $\alpha$ is given by \eqref{alpha_n}.
\end{theorem}
\begin{remark}
Estimate \eqref{grad_u_lower_bound} holds for all $C^2$ domains $\Omega$ and $C^{4}$ relatively strictly convex open sets $D_1, D_2$ that are axially symmetric with respect to $x_n$-axis. A modification of the proof of the theorem yields the result.
\end{remark}

Let us give a brief description of the proof of Theorem \ref{optimality}. Consider
$$\bar{u}(x') =  \fint_{-\varepsilon/2 + g(x') <x_n<\varepsilon/2 + f(x')} u(x',x_n) \, dx_n, \quad |x'| < 1,$$
where $f(x') =  -\sqrt{1-|x'|^2} + 1$ and $g(x') = \sqrt{1-|x'|^2} - 1$.
In the polar coordinates, $\bar u (x') = \bar u(r, \xi)$, where $x' = (r,\xi)$, $0 < r < 1$, and $\xi \in \bS^{n-2}$. Since the boundary value $\varphi$ depends only on $x_1$ and is odd in $x_1$, the projection of $\bar u(r , \cdot)$ to the span of the spherical harmonics is $U_{1,1}(r)Y_{1,1}(\xi)$, where $Y_{1,1}$ is $\left.x_1\right|_{\bS^{n-2}}$ modulo a harmless positive normalization constant,
$$
U_{1,1}(r)= \fint_{-\varepsilon/2 + g(x') <x_n<\varepsilon/2 + f(x')} \hat{u}(r,x_n) \, dx_n,
$$
and
$$
\hat{u}(r,x_n) = \int_{\bS^{n-2}} u(r,\xi,x_n)Y_{1,1}(\xi) \, d \xi.
$$
We analyze the equations satisfied by $U_{1,1}(r)$ and $\hat{u}(r,x_n)$ and establish a lower bound
$$
U_{1,1}(r) \ge \frac{1}{C} r^\beta (\varepsilon + r^2)^{\frac{\alpha - \beta}{2}},\quad 0<r <1,
$$
where $\beta = \frac{2\alpha^2 + \alpha(n-1)}{n-3+\alpha}$ and $C$ is a positive constant independent of $\varepsilon$. It follows that
$$
\| \bar{u}(\sqrt{\varepsilon},\cdot) \|_{L^2(\bS^{n-2})} \ge |U_{1,1}(\sqrt{\varepsilon})| \ge \frac{1}{C} \varepsilon^{\alpha/2},
$$
and, consequently, there exists $\xi_0 \in \bS^{n-2}$, $x_n \in (-\varepsilon/2 + g(x'), \varepsilon/2 + f(x'))$ such that
$$
|u(\sqrt{\varepsilon}, \xi_0, x_n)| \ge \frac{1}{C} \varepsilon^{\alpha/2}.
$$
Estimate \eqref{grad_u_lower_bound} follows since $u(0) = 0$ by the oddness of $u$ in $x_1$.

Theorems \ref{radial_thm} and \ref{optimality} will be proved in Sections 2 and 3, respectively.

\section{Proof of Theorem \ref{radial_thm}}
In this section, we prove Theorem \ref{radial_thm}. Without loss of generality, we may assume $a = 1$. Namely, we consider
\begin{equation*}
\label{fg_assumption_1}
f(x')- g(x') =|x'|^2 + O(|x'|^{2+\gamma}) \quad\mbox{for}~~0<|x'|<R_{0}.
\end{equation*}
We perform a change of variables by setting
\begin{equation}\label{x_to_y}
\left\{
\begin{aligned}
y' &= x' ,\\
y_n &= 2 \varepsilon \left( \frac{x_n - g(x') + \varepsilon/2}{\varepsilon + f(x') - g(x')} - \frac{1}{2} \right),
\end{aligned}\right.
\quad \forall (x',x_n) \in \Omega_{R_0}.
\end{equation}
This change of variables maps the domain $\Omega_{R_0}$ to a cylinder of height $\varepsilon$, denoted by $Q_{R_0, \varepsilon}$, where
\begin{equation*}\label{Q_s_t}
Q_{s,t}:= \{ y = (y',y_n) \in \bR^n ~\big|~  |y'| < s,  |y_n| < t\}
\end{equation*}
for $s,t > 0$. Moreover, $\det (\partial_x y)=2\varepsilon(\varepsilon + f(x') - g(x'))^{-1}$.
Let $u(x) \in H^1(\Omega_{R_0})$ be a solution of \eqref{main_problem_narrow} and let $v(y) = u(x)$. Then $v$ satisfies
\begin{equation}\label{equation_v}
\left\{
\begin{aligned}
-\partial_i(a^{ij}(y) \partial_j v(y)) &=0 \quad \mbox{in } Q_{R_0, \varepsilon},\\
a^{nj}(y) \partial_j v(y) &= 0 \quad \mbox{on } \{y_n = -\varepsilon\} \cup \{y_n = \varepsilon\},
\end{aligned}
\right.
\end{equation}
with $\|v\|_{L^\infty(Q_{R_0, \varepsilon})} \le 1$,
where the coefficient matrix $(a^{ij}(y))$ is given by
\begin{align}\label{a_ij_formula}
&(a^{ij}(y)) = \frac{ 2 \varepsilon(\partial_x y)(\partial_x y)^t}{\det (\partial_x y)}\nonumber\\
&= \begin{pmatrix}
\varepsilon + |y'|^2 &0 &\cdots &0 &a^{1n}\\
0 &\varepsilon + |y'|^2 &\cdots &0 &a^{2n}\\
\vdots &\vdots &\ddots &\vdots &\vdots\\
0 &0 &\cdots &\varepsilon + |y'|^2 &a^{n-1,n}\\
a^{n1} &a^{n2} &\cdots &a^{n,n-1} & \frac{4\varepsilon^2 + \sum_{i=1}^{n-1}|a^{in}|^2}{\varepsilon + |y'|^2}
\end{pmatrix}
+
\begin{pmatrix}
e^1 &0 &\cdots &0\\
0 &e^2 &\cdots &0\\
\vdots &\vdots &\ddots &\vdots\\
0 &0 &\cdots & e^n
\end{pmatrix},
\end{align}
and
$$
a^{ni} = a^{in} = -2\varepsilon \partial_i g(y') - (y_n + \varepsilon)\partial_i(f(y') - g(y'))
$$
for $i = 1, \ldots , n-1$. By \eqref{fg_0}, we know that for $i = 1, \ldots , n-1$,
\begin{equation}
\label{coefficient_bounds}
|a^{ni}(y)| = |a^{in}(y)| \le C \varepsilon |y'|, \quad |e^i(y')| \le C|y'|^{2+\gamma}, \quad \mbox{and} \quad |e^n(y)| \le \frac{C\varepsilon^2|y'|^\gamma}{\varepsilon + |y'|^2}.
\end{equation}
Note that $e^1, \ldots , e^{n-1}$ depend only on $y'$ and are independent of $y_n$. We define
\begin{equation}
\label{v_bar_def}
\bar{v}(y') := \fint_{-\varepsilon}^\varepsilon v(y',y_n)\, dy_n.
\end{equation}
Then $\bar{v}$ satisfies
\begin{equation}\label{equation_v_bar}
\dv((\varepsilon+|y'|^2)\nabla \bar v)=-\sum_{i=1}^{n-1}\partial_i(\overline{a^{in}\partial_n v}) - \sum_{i=1}^{n-1} \partial_i(e^i \partial_i \bar v) \quad\text{in}\,\,B_{R_0}\subset \bR^{n-1},
\end{equation}
with $\| \bar{v} \|_{L^\infty(B_{R_0})} \le 1$, where $\overline{a^{in}\partial_n v}$ is the average of $a^{in}\partial_n v$ with respect to $y_n$ in $(-\varepsilon,\varepsilon)$. Since $\frac{\partial u}{\partial \nu} = 0$ on $\Gamma_+$ and $\Gamma_-$, we have, by
\eqref{fg_0} and \eqref{grad_u_bound_rough} that
$$|\partial_n u(x)| \le C\left| \sum_{i = 1}^{n-1}x_i \partial_i u \right| \le C, \quad \forall x \in \Gamma_+ \cup \Gamma_-.$$
By the harmonicity of $\partial_n u$, the estimate \eqref{u_C1_outside}, and the maximum principle,
\begin{equation}
\label{grad_n_u_bound}
|\partial_n u| \le C \quad \mbox{in}~~\Omega_{R_0},
\end{equation}
and consequently,
\begin{equation}
\label{grad_n_v_bound}
|\partial_n v| \le C(\varepsilon+|y'|^2)/\varepsilon \quad \mbox{in}~~Q_{R_0,\varepsilon}.
\end{equation}
Therefore, the equation \eqref{equation_v_bar} can be rewritten as
\begin{equation}
\label{equation_v_bar_2}
\dv((\varepsilon+|y'|^2)\nabla \bar v)=\sum_{i=1}^{n-1}\partial_i F_i\quad\text{in}\,\,B_{R_0}\subset \bR^{n-1},
\end{equation}
where $F_i:=-\overline{a^{in}\partial_n v} - e^i\partial_i \bar{v}$ satisfies, using \eqref{grad_u_bound_rough} and \eqref{coefficient_bounds},
\begin{equation*}
|F_i|\le C\left(|y'|(\varepsilon + |y'|^2) + |y'|^{2+\gamma}(\varepsilon+|y'|^2)^{-1/2}\right).
\end{equation*}
For $\gamma, s \in \bR$, we introduce the following norm
$$
\| F \|_{\varepsilon, \gamma,s,B_{R_0}}: = \sup_{y' \in B_{R_0}} \frac{|F(y')|}{|y'|^{\gamma} (\varepsilon + |y'|^2)^{1-s}}.
$$

\begin{proposition}
\label{prop_grad_v_bar_control}
For $n \ge 3$, $s \ge 0$, $1+\gamma - 2s > 0$, $1+\gamma - 2s \neq \alpha$, $\varepsilon> 0$, and $R_0 > 0$,
let $\bar{v} \in H^1(B_{R_0})$ be a solution of
\begin{equation*}
\label{equation_v_bar_3}
\dv((\varepsilon+|y'|^2)\nabla \bar v)= \dv F + G \quad\text{in}\,\,B_{R_0}\subset \bR^{n-1},
\end{equation*}
where $F,G \in L^\infty(B_{R_0})$ satisfy
\begin{equation}
\label{FG_bound}
\|F\|_{\varepsilon, \gamma,s,B_{R_0}}< \infty , \quad \|G\|_{\varepsilon, \gamma-1,s,B_{R_0}} < \infty.
\end{equation}
Then for any $R \in (0, R_0/2)$, we have
\begin{equation}
\label{v_bar_L2_sphere}
\Big(\fint_{\partial B_R}|\bar v - \bar{v}(0)|^2\,d\sigma\Big)^{1/2} \le C(\|F\|_{\varepsilon, \gamma,s,B_{R_0}} + \|G\|_{\varepsilon, \gamma-1,s,B_{R_0}} +  \| \bar v - \bar v(0) \|_{L^2(\partial B_{R_0})})  R^{\tilde\alpha},
\end{equation}
where $\tilde\alpha := \min \{ \alpha , 1 + \gamma - 2s \}$, $\alpha$ is given in \eqref{alpha_n}, and $C$ is some positive constant depending only on $n$, $\gamma$, $s$, and $R_0$, and is independent of $\varepsilon$.
\end{proposition}

For the proof, we use an iteration argument based on the following two lemmas.

\begin{lemma}
\label{lemma_v1}
For $n \ge 3$, $\varepsilon > 0$, and $R_0 > 0$, let $v_1 \in H^1(B_{R_0})$ satisfy
\begin{equation}
\label{equation_v1}
\dv((\varepsilon+|y'|^2)\nabla v_1)=0\quad\text{in}\,\,B_{R_0} \subset \bR^{n-1}.
\end{equation}
Then for any $0 < \rho < R \le R_0$, we have
$$
\left( \fint_{\partial B_\rho} |v_1(y') - v_1(0)|^2 \, d\sigma \right)^{\frac{1}{2}} \le \left(\frac{\rho}{R} \right)^{\alpha} \left( \fint_{\partial B_R} |v_1(y') - v_1(0)|^2 \, d\sigma \right)^{\frac{1}{2}},
$$
where $\alpha$ is given in \eqref{alpha_n}.
\end{lemma}
\begin{proof}
By the elliptic theory, $v_1 \in C^\infty(B_{R_0})$. Without loss of generality, we assume that $v_1(0) = 0$. By scaling, it suffices to prove the lemma for $R = 1$. Denote $y' = (r ,\xi) \in (0,1) \times \bS^{n-2}$. We can rewrite \eqref{equation_v1} as
$$
\partial_{rr} v_1 +\left(\frac {n-2}r +\frac{2r}{\varepsilon+r^2} \right)\partial_{r} v_1 +\frac 1 {r^2}\Delta_{\bS^{n-2}} v_1=0 \quad \mbox{in}~~B_1 \setminus \{0\}.
$$
Take the decomposition
\begin{equation}
\label{v1_expansion}
v_1(y') = \sum_{k=1}^\infty \sum_{i=1}^{N(k)} V_{k,i}(r)Y_{k,i}(\xi), \quad y' \in B_1\setminus\{0\},
\end{equation}
where $Y_{k,i}$ is a $k$-th degree spherical harmonics, that is,
$$
-\Delta_{\bS^{n-2}} Y_{k,i} = k(k+n-3)Y_{k,i}
$$ and $\{Y_{k,i}\}_{k,i}$ forms an orthonormal basis of $L^2(\bS^{n-2})$.
Here we used the fact that $V_{0,1}=0$ because $v_1(0)=0$.
Then $V_{k,i}(r) \in C^2(0,1)$ is given by
$$
V_{k,i}(r) = \int_{\bS^{n-2}} v_1(y') Y_{k,i}(\xi) \, d\xi,
$$
and satisfies
\begin{equation*}
L_k V_{k,i}:= V_{k,i}''(r) +\left(\frac {n-2}r +\frac{2r}{\varepsilon+r^2} \right)V_{k,i}'(r) - \frac {k(k+n-3)} {r^2} V_{k,i}(r)=0 \quad \mbox{in}~~(0,1)
\end{equation*}
for each $k \in \bN$, $i = 1, 2 ,\ldots, N(k)$. For any $k \in \bN$, let
$$
\alpha_k:= \frac {-(n-1)+\sqrt{(n-1)^2+4k(k+n-3)}}{2}.
$$
For any $c \in \bR$, we have, by a direct computation,
\begin{align*}
L_kr^{c}=r^{c-2}\Big(c^2+(n-3+2r^2/(\varepsilon+r^2) )c -k(k+n-3)\Big) \quad \mbox{in}~~(0,1).
\end{align*}
Thus for $c > 0$ sufficiently small, we have
$$
L_kr^{-c}\le 0 \quad
\mbox{and} \quad
L_k r^{\alpha_k}\le 0 \quad \mbox{in}~~(0,1).
$$
Therefore, for any $\gamma>0$,
$$
L_k(\pm V_{k,i}(r)-\gamma r^{-c} - |V_{k,i}(1)|r^{\alpha_k})\ge 0 \quad \mbox{in}~~(0,1).
$$
Since $v_1 \in L^\infty(B_{1})$, we know that $V_{k,i}(r)$ is bounded in $(0,1)$, so we have
$$
\pm V_{k,i}(r)-\gamma r^{-c} - |V_{k,i}(1)|r^{\alpha_k} < 0 \quad \mbox{as}~~r \searrow 0.
$$
Clearly,
$$
\pm V_{k,i}(r)-\gamma r^{-c} - |V_{k,i}(1)|r^{\alpha_k} < 0 \quad \mbox{when}~~r =1.
$$
By the maximum principle,
$$
|V_{k,i}(r)|\le \gamma r^{-c} + r^{\alpha_k} |V_{k,i}(1)| \quad \mbox{for}~~0<r<1.
$$
Sending $\gamma \to 0$, we have
\begin{equation}
\label{V_decay}
|V_{k,i}(r)|\le r^{\alpha_k} |V_{k,i}(1)| \quad \mbox{for}~~0<r<1.
\end{equation}
It follows from \eqref{v1_expansion} and \eqref{V_decay},
\begin{align*}
\fint_{\partial B_\rho} |v_1(y')|^2 \, d\sigma &= \sum_{k=1}^\infty \sum_{i=1}^{N(k)}  |V_{k,i}(\rho)|^2 \\
&\le \rho^{2\alpha} \sum_{k=1}^\infty \sum_{i=1}^{N(k)}|V_{k,i}(1)|^2 \,
= \rho^{2\alpha} \fint_{\partial B_1} |v_1(y')|^2 \, d\sigma.
\end{align*}
\end{proof}

\begin{lemma}
\label{lemma_v2}
For $n \ge 3$, $s \ge 0$, $1+\gamma - 2s > 0$, and $\varepsilon > 0$, suppose that $F,G \in L^\infty(B_{1})$ satisfy \eqref{FG_bound} with $R_0 =1$, and $v_2 \in H^1_0(B_{1})$ satisfies
\begin{equation}
\label{equation_v2}
\dv((\varepsilon+|y'|^2)\nabla v_2)=\dv F + G\quad\text{in}\,\,B_{1}\subset \bR^{n-1}.
\end{equation}
Then we have
$$
\|v_2\|_{L_\infty(B_{1})}\le C(\|F\|_{\varepsilon,\gamma,s, B_{1}} + \|G\|_{\varepsilon,\gamma-1,s,B_{1}}),
$$
where $C>0$ depends only on $n$, $\gamma$, and $s$, and is in particular independent of $\varepsilon$.
\end{lemma}
\begin{proof}
Without loss of generality, we assume $\|F\|_{\varepsilon,\gamma,s,B_1}+ \|G\|_{\varepsilon,\gamma-1,s,B_{1}} = 1$. Denote $r = |y'|$. We can rewrite \eqref{equation_v2} as
\begin{equation}
\label{equation_v2_2}
\Delta v_2+\frac{2r}{\varepsilon+r^2}\partial_r v_2=\partial_i (F_i(\varepsilon+r^2)^{-1})+2F_iy_i(\varepsilon+r^2)^{-2} + G(\varepsilon + r^2)^{-1}\quad\text{in}\,\,B_{1}.
\end{equation}
We use Moser's iteration argument. By the definitions,
\begin{align*}
|F_i(\varepsilon+r^2)^{-1}|&\le  r^{\gamma-2s}\|F\|_{\varepsilon,\gamma,s,B_1},\\
|F_iy_i(\varepsilon+r^2)^{-2}| &\le  r^{\gamma-2s-1}\|F\|_{\varepsilon,\gamma,s,B_1},\\
|G(\varepsilon + r^2)^{-1}| &\le  r^{\gamma-2s-1}\|G\|_{\varepsilon,\gamma-1,s,B_1}.
\end{align*}
For $p \ge 2$, we multiply the equation \eqref{equation_v2_2} with $-|v_2|^{p-2}v_2$ and integrate by parts to obtain
\begin{align*}
&(p-1) \int_{B_{1}} |\nabla v_2|^{2}|v_2|^{p-2}\,dy'
-\int_{B_{1}}\frac{2r}{\varepsilon+r^2}\partial_r v_2(|v_2|^{p-2}v_2) \,dy'\\
&\le C(p-1)\int_{B_{1}}|\nabla v_2||v_2|^{p-2} r^{\gamma-2s}\,dy' + C \int_{B_{1}} |v_2|^{p-1} r^{\gamma - 2s - 1}.
\end{align*}
The second term on the left-hand side is equal to
\begin{align*}
-\frac 1 p\int_{\bS^{n-2}} \int_{0}^{1}\frac{2r^{n-1}}{\varepsilon+r^2}\partial_r |v_2|^{p}  \,drd\theta =\frac {1} p \int_{\bS^{n-2}}\int_{0}^{1}\partial_r \Big(\frac{2r^{n-1}}{\varepsilon+r^2}\Big) |v_2|^{p}\,dr d\theta \ge 0.
\end{align*}
Therefore, by H\"older's inequality and using $1+\gamma-2s > 0$,
\begin{align*}
&(p-1) \int_{B_{1}} |\nabla v_2|^{2}|v_2|^{p-2}\,dy'\nonumber\\
&\le C(p-1) \||\nabla v_2||v_2|^{\frac{p-2}{2}} \|_{L^2(B_{1})} \|v_2^{p-2}\|_{L^{\frac{n-1+2\delta}{n-3+2\delta}}(B_{1})}^{1/2} \| r^{2(\gamma - 2s)} \|_{L^{\frac{n-1}{2} + \delta}(B_{1})}^{1/2}\nonumber\\
 &\quad + C\| v_2^{p-1}\|_{L^{\frac{n-1+2\delta}{n-3+2\delta}}(B_{1})}  \| r^{\gamma - 2s -1} \|_{L^{\frac{n-1}{2} + \delta}(B_{1})},
\end{align*}
where $\delta > 0$ is chosen sufficiently small so that
$$ \| r^{2(\gamma - 2s)} \|_{L^{\frac{n-1}{2} + \delta}(B_{1})} +  \| r^{\gamma - 2s -1} \|_{L^{\frac{n-1}{2} + \delta}(B_{1})} < \infty.$$
By Young's inequality,
\begin{align}
&\frac{p-1}{2} \int_{B_{1}} |\nabla v_2|^{2}|v_2|^{p-2}\,dy'\notag \\
&\le C(p-1) \|v_2^{p-2}\|_{L^{\frac{n-1+2\delta}{n-3+2\delta}}(B_{1})} \| r^{2(\gamma - 2s)} \|_{L^{\frac{n-1}{2} + \delta}(B_{1})}\nonumber\\
&\quad + C\| v_2^{p-1}\|_{L^{\frac{n-1+2\delta}{n-3+2\delta}}(B_{1})}  \| r^{\gamma - 2s -1} \|_{L^{\frac{n-1}{2} + \delta}(B_{1})}. \label{moser_inequality1}
\end{align}
Taking $p = 2$ in the above, we have, by H\"older's inequality,
\begin{align*}
\int_{B_{1}} |\nabla v_2|^{2} dy' &\le C+ C \| v_2\|_{L^{\frac{n-1+2\delta}{n-3+2\delta}}(B_{1})} \le C + C \|v_2\|_{L^{\frac{2(n-1+2\delta)}{n-3+2\delta}}(B_{1})}.
\end{align*}
Applying the Sobolev-Poincar\'e inequality on the left-hand side, we have
\begin{equation}
\label{starting_point}
\|v_2\|_{L^{\frac{2(n-1+2\delta)}{n-3+2\delta}}(B_{1})} \le C.
\end{equation}
From \eqref{moser_inequality1}, by H\"older's inequality,
\begin{align*}
\frac{4(p-1)}{p^2} \int_{B_{1}}\left| \nabla |v_2|^{\frac{p}{2}} \right|^2\,dy' &= (p-1)\int_{B_{1}}|\nabla v_2|^{2}|v_2|^{p-2}\,dy'\\
&\le Cp \| v_2\|^{p-2}_{L^{\frac{(n-1+2\delta)p}{n-3+2\delta}}(B_{1})} + C\| v_2\|^{p-1}_{L^{\frac{(n-1+2\delta)p}{n-3+2\delta}}(B_{1})},
\end{align*}
which implies that
\begin{align*}
\|\nabla |v_2|^{\frac{p}{2}}\|_{L^2(B_1)}^{2/p} \le \max_{i \in \{1,2\}} \left( Cp^{i} \right)^{1/p} \| v_2\|^{(p-i)/p}_{L^{\frac{(n-1+2\delta)p}{n-3+2\delta}}(B_{1})}.
\end{align*}
Then by the Sobolev inequality 
and Young's inequality, we have
\begin{align*}
\|v_2\|_{L^{tp}(B_{1})} \le & \max_{i \in \{1,2\}} \left( Cp^i \right)^{1/p} \left(\frac{p-i}{p} \|v_2\|_{L^{\frac{(n-1+2\delta)p}{n-3+2\delta}}(B_{1})} + \frac{i}{p} \right)\\
\le &  \left( Cp^2 \right)^{1/p} \left( \|v_2\|_{L^{\frac{(n-1+2\delta)p}{n-3+2\delta}}(B_{1})} + \frac{2}{p} \right)
  \quad \mbox{when}~~n=3
\end{align*}
for any $t > \frac{n-1+2\delta}{n-3+2\delta}$
and,
$$
\|v_2\|_{L^{\frac{(n-1)p}{n-3}}(B_{1})} \le \left( Cp^2 \right)^{1/p} \left( \|v_2\|_{L^{\frac{(n-1+2\delta)p}{n-3+2\delta}}(B_{1})} + \frac{2}{p} \right) \quad \mbox{when}~~n > 3.
$$
For $k \ge 0$, let
$$
p_k = 2\left( t \frac{n-3+2\delta}{(n-1+2\delta)} \right)^k \frac{n-1+2\delta}{n-3+2\delta} \quad \mbox{when}~~n = 3,
$$
and
$$
p_k = 2 \left( \frac{n-1}{n-3} \frac{n-3+2\delta}{(n-1+2\delta)} \right)^k \frac{n-1+2\delta}{n-3+2\delta} \quad \mbox{when}~~n > 3.
$$ Iterating the relations above, we have, by \eqref{starting_point},
\begin{align}
\label{Moser_iteration}
\|v_2\|_{L^{p_k}} &\le \prod_{i=0}^{k-1} \left( Cp_i^2 \right)^{1/p_i}\|v_2\|_{L^{p_0}(B_{1})} + \sum_{i=0}^{k-1} \prod_{j=0}^{k-1-i} \left( Cp_{k-1-j}^2 \right)^{1/p_{k-1-j}} \frac{2}{p_i} \nonumber\\
&\le C \|v_2\|_{L^{\frac{2(n-1+2\delta)}{n-3+2\delta}}(B_{1})} + C \le C,
\end{align}
where $C$ is a positive constant depending on $n$, $\gamma$, and $s$, and is in particular independent of $k$. The lemma is concluded by taking $k \to \infty$ in \eqref{Moser_iteration}.
\end{proof}

\begin{proof}[Proof of Proposition \ref{prop_grad_v_bar_control}]
Without loss of generality, we assume that $\bar{v}(0) = 0$ and $$
\|F\|_{\varepsilon, \gamma,s,B_{R_0}}+\|G\|_{\varepsilon, \gamma-1,s,B_{R_0}} +  \| \bar v  \|_{L^2(\partial B_{R_0})}=1.
$$
Consider
$$
\omega(\rho):=\Big(\fint_{\partial B_\rho}|\bar v|^2\,d\sigma\Big)^{1/2}.
$$
For $0 < \rho \le R/2 \le R_0/2$, we write $\bar v=v_1+v_2$ in $B_R$, where $v_2$ satisfies
$$
\dv((\varepsilon+|y'|^2)\nabla v_2)=\dv F + G\quad\text{in}\,\,B_R
$$
and $v_2=0$ on $\partial B_R$. Thus $v_1$ satisfies
$$
\dv((\varepsilon+|y'|^2)\nabla v_1)=0\quad\text{in}\,\,B_R,
$$
and $v_1=\bar v$ on $\partial B_R$. By Lemma \ref{lemma_v1},
\begin{equation}
\label{v1_control}
\left( \fint_{\partial B_\rho} |v_1(y') - v_1(0)|^2 \, d\sigma \right)^{\frac{1}{2}} \le \left(\frac{\rho}{R} \right)^{\alpha} \left( \fint_{\partial B_R} |v_1(y') - v_1(0)|^2 \, d\sigma \right)^{\frac{1}{2}}.
\end{equation}
Since $\tilde v_2(y'):=v_2(Ry')$ satisfies
$$
\dv((R^{-2}\varepsilon+|y'|^2)\nabla \tilde v_2)
=\dv \tilde F + \tilde G\quad\text{in}\,\,B_1,
$$
where $\tilde F(y'):=R^{-1}F(Ry')$ and $\tilde{G}(y') := G(Ry')$ satisfy
\begin{align*}
\|\tilde{F}\|_{R^{-2}\varepsilon,\gamma,s,B_1} &= R^{1+ \gamma - 2s} \|F\|_{\varepsilon, \gamma,s,B_{R}}, \\
\|\tilde G\|_{R^{-2}\varepsilon, \gamma-1,s,B_{1}} &= R^{1+ \gamma - 2s} \|G\|_{\varepsilon, \gamma-1,s,B_{R}},
\end{align*}
we apply Lemma \ref{lemma_v2} to $\tilde{v}_2$ with $\varepsilon$ replaced with $R^{-2}\varepsilon$ to obtain
\begin{equation}
\label{v2_control}
\|v_2\|_{L_\infty(B_R)}\le CR^{1+\gamma-2s}.
\end{equation}
Since $\bar{v}(0) = v_1(0) + v_2(0) = 0$, we have $|v_1(0)|=|v_2(0)|$. Combining \eqref{v1_control} and \eqref{v2_control} yields, using $\bar{v} = v_1 + v_2$, and $\bar{v} = v_1$ on $\partial B_R$,
\begin{align}
\label{omega_iteration}
\omega(\rho)&\le \left( \fint_{\partial B_\rho} |v_1(y') - v_1(0)|^2 \, d\sigma \right)^{\frac{1}{2}} + \left( \fint_{\partial B_\rho} |v_2(y') - v_2(0)|^2 \, d\sigma \right)^{\frac{1}{2}}\nonumber\\
&\le \left(\frac{\rho}{R} \right)^{\alpha} \left( \fint_{\partial B_R} |v_1(y')|^2 \, d\sigma \right)^{\frac{1}{2}} + \left(\frac{\rho}{R} \right)^{\alpha} |v_1(0)| + 2\|v_2\|_{L_\infty(B_R)}\nonumber\\
&\le \left(\frac{\rho}{R} \right)^{\alpha} \omega(R) + CR^{1+ \gamma - 2s}.
\end{align}
For a positive integer $k$, we take $\rho = 2^{-i-1}R_0$ and $R = 2^{-i}R_0$ in \eqref{omega_iteration} and iterate from $i = 0$ to $k-1$. We have, using $1 + \gamma - 2s \neq \alpha$,
\begin{align*}
\omega(2^{-k}R_0) &\le 2^{-k\alpha} \omega(R_0) + C\sum_{i=1}^k 2^{-(k-i)\alpha} (2^{1-i}R_0)^{1+\gamma - 2s}\\
&\le 2^{-k\alpha} \omega(R_0) + C2^{-k\alpha} R_0^{1+\gamma - 2s} \frac{1 - 2^{k(\alpha -1-\gamma+2s)}}{1 - 2^{\alpha - 1 - \gamma + 2s}}.
\end{align*}
It follows that
$$
\omega(2^{-k}R_0) \le 2^{-k\tilde\alpha} \left( \omega(R_0) + CR_0^{1 + \gamma - 2s}\right).
$$
For any $\rho \in (0, R_0/2)$, let $k$ be the integer such that $2^{-k-1}R_0 < \rho \le 2^{-k}R_0$. Then
\begin{equation*}
\omega(\rho) \le C \rho^{\tilde\alpha}, \quad \forall \rho \in (0, R_0/2).
\end{equation*}
Therefore, \eqref{v_bar_L2_sphere} is proved.
\end{proof}

\begin{proof}[Proof of Theorem \ref{radial_thm}]
Without loss of generality, we assume that $a= 1$, $u(0) = 0$ and $\|u\|_{L^\infty(\Omega_{R_0})} = 1$. We make the change of variables \eqref{x_to_y}, and let $v(y) = u(x)$. Then $v$ satisfies \eqref{equation_v}. Let $\bar{v}$ be defined as in \eqref{v_bar_def}. By \eqref{grad_u_bound_rough},
$$\|\nabla \bar{v}(y')\|_{\varepsilon,0,s_0+1,B_{R_0}} < \infty,$$
where $s_0 = \frac{1}{2}$. Then $\bar{v}$ satisfies the equation \eqref{equation_v_bar_2} with $F$ satisfying
$$\|F\|_{\varepsilon, \gamma - 2s_0, 0, B_{R_0}} < \infty.$$
By \eqref{grad_n_v_bound},
\begin{equation}
\label{v-v_bar}
|v(y', y_n) - \bar{v}(y')| \le 2\varepsilon \max_{y_n\in (- \varepsilon , \varepsilon)} |\partial_n v(y',y_n)| \le C (\varepsilon + |y'|^2) \quad \mbox{in}~~Q_{R_0,\varepsilon}.
\end{equation}
By decreasing $\gamma$ if necessary, we may assume that $1+\gamma-2s_0=\gamma<\alpha$.
By Proposition \ref{prop_grad_v_bar_control} and \eqref{v-v_bar}, we have
\begin{align*}
\fint_{Q_{2\varepsilon^{1/2} , \varepsilon}} |v - \bar{v}(0)|^2 \, dy &\le C\fint_{Q_{2\varepsilon^{1/2} , \varepsilon}} |v - \bar{v}|^2 \, dy + C \fint_{Q_{2\varepsilon^{1/2} , \varepsilon}} |\bar{v} - \bar{v}(0)|^2 \, dy \le C \varepsilon^{\tilde\alpha},
\end{align*}
where $\tilde{\alpha} = \min\{\alpha , 1+ \gamma - 2s_0 \}$.
Let $\tilde{a}^{ij}(y) = \varepsilon^{-1} a^{ij}(\varepsilon^{1/2} y)$ and $\tilde{v}(y) = v(\varepsilon^{1/2} y) - \bar{v}(0)$. Then $\tilde{v}$ satisfies
\begin{equation*}
\left\{
\begin{aligned}
-\partial_i(\tilde a^{ij}(y) \partial_j \tilde v(y)) &=0 \quad \mbox{in } Q_{2, \varepsilon^{1/2}},\\
\tilde a^{nj}(y) \partial_j \tilde v(y) &= 0 \quad \mbox{on } \{y_n = -\varepsilon^{1/2}\} \cup \{y_n = \varepsilon^{1/2}\}.
\end{aligned}
\right.
\end{equation*}
For any $\mu \in(0,1)$, it is straightforward to verify that
$$
\frac{I}{C} \le \tilde{a} \le CI \quad \mbox{and} \quad \| \tilde{a} \|_{C^{\mu} (Q_{2,\varepsilon^{1/2}})} \le C.
$$
Now we define
$$
S_l:= \left\{y \in \bR^n ~\big|~  |y'| < 2,~ (2l-1) \varepsilon^{1/2} < y_n < (2l+1) \varepsilon^{1/2} \right\}
$$
for any integer $l$, and
$$
S: = \left\{y \in \bR^n ~\big|~  |y'| < 2,~ |y_n| < 2\right\}.
$$
Note that $Q_{2, \varepsilon^{1/2}} = S_0$. We take the even extension of $\tilde v$ with respect to $y_n=\varepsilon^{1/2}$ and then take the periodic extension (so that the period is equal to $4\varepsilon^{1/2}$).
More precisely,
we define, for any $l \in \bZ$, a new function $\hat{v}$ by setting
$$\hat{v}(y) := \tilde{v}\left(y', (-1)^l\left(y_n - 2l \varepsilon^{1/2}\right)\right), \quad \forall y \in S_l.$$
We also define the corresponding coefficients, for $k = 1,2, \ldots, n-1$,
$$\hat{a}^{nk}(y)=\hat{a}^{kn}(y) := (-1)^l\tilde{a}^{kn}\left(y', (-1)^l\left(y_n - 2l \varepsilon^{1/2}\right)\right),  \quad \forall y \in S_l,$$
and for other indices,
$$\hat{a}^{ij}(y) := \tilde{a}^{ij}\left(y', (-1)^l\left(y_n - 2l \varepsilon^{1/2}\right)\right), \quad \forall y \in S_l.$$
Then $\hat{v}$ and $\hat{a}^{ij}$ are defined in the infinite cylinder $Q_{2, \infty}$. In particular, by using the conormal boundary conditions, it is easily seen that $\hat{v}$ satisfies the equation
$$
\partial_i (\hat{a}^{ij} \partial_j \hat{v}) = 0 \quad \mbox{in}\,\,S.
$$
By \cite{LN}*{Proposition 4.1} and \cite{LY2}*{Lemma 2.1},
$$
\| \nabla \hat{v} \|_{L^\infty(\frac{1}{2}S)} \le C \| \hat{v}\|_{L^2(S)} \le C \varepsilon^{\frac{\tilde\alpha}{2}},
$$
which implies, after reversing the changes of variables,
$$
\| \nabla u\|_{L^\infty (\Omega_{\varepsilon^{1/2}})} \le C\varepsilon^{\frac{\tilde\alpha-1}{2}}.
$$
For any $R \in (\varepsilon^{1/2}, R_0/4)$, by Proposition \ref{prop_grad_v_bar_control} and \eqref{v-v_bar}, we have
$$
\fint_{Q_{4R, \varepsilon} \setminus Q_{R/2, \varepsilon}} |v - \bar{v}(0)|^2 \, dy \le C R^{2\tilde\alpha}.
$$
This implies
$$
\fint_{\Omega_{4R} \setminus \Omega_{R/2}} |u - \bar{v}(0)|^2 \, dx \le C R^{2\tilde\alpha}.
$$
We make a change of variables by setting
\begin{equation*}\label{x_to_z}
\left\{
\begin{aligned}
z' &= x' ,\\
z_n &=  2 R^2 \left( \frac{x_n - g(x') + \varepsilon/2}{\varepsilon + f(x') - g(x')} - \frac{1}{2} \right),
\end{aligned}\right.
\quad \forall (x',x_n) \in \Omega_{4R} \setminus \Omega_{R/2}.
\end{equation*}
This change of variables maps the domain $\Omega_{4R} \setminus \Omega_{R/2}$ to $Q_{4R, R^2} \setminus Q_{R/2 , R^2}$. Let $w(z) = u(x) - \bar{v}(0)$, so that $w(z)$ satisfies
\begin{equation*}
\left\{
\begin{aligned}
-\partial_i(b^{ij}(z) \partial_j w(z)) &=0 \quad \mbox{in } Q_{4R, R^2} \setminus Q_{R/2 , R^2},\\
 b^{nj}(z) \partial_j w(z) &= 0 \quad \mbox{on } \{z_n = -R^2\} \cup \{z_n = R^2\},
\end{aligned}
\right.
\end{equation*}
where
$$
(b^{ij}(z)) = \frac{(\partial_x z)(\partial_x z)^t}{\det (\partial_x z)}.
$$
It is straightforward to verify that
$$
\frac{I}{C} \le b(z) \le CI \quad \mbox{and}\quad \|b\|_{C^\mu (Q_{4R, R^2} \setminus Q_{R/2 , R^2})} \le CR^{-\mu}
$$ for any $\mu > 0$. We can apply the ``flipping argument" as above to get
$$
\| \nabla w \|_{L^\infty(Q_{2R, R^2} \setminus Q_{R , R^2})} \le CR^{\tilde\alpha-1},
$$
which implies
$$
\| \nabla u\|_{L^\infty(\Omega_{2R} \setminus \Omega_R)} \le CR^{\tilde\alpha-1}
$$
for any $R \in (\varepsilon^{1/2}, R_0/4)$. Therefore, we have improved the upper bound $|\nabla u(x)| \le C(\varepsilon + |x'|^2)^{-s_0}$ to $|\nabla u(x)| \le C(\varepsilon + |x'|^2)^{\frac{\tilde{\alpha}-1}{2}}$, where $\frac{\tilde{\alpha}-1}{2} = \min\left\{\frac{\alpha - 1}{2}, -s_0 + \frac{\gamma}{2} \right\}$. If $-s_0 + \frac{\gamma}{2} < \frac{\alpha - 1}{2}$, we take $s_1 = s_0 - \frac{\gamma}{2}$ and repeat the argument above. We may decrease $\gamma$ if necessary so that $\frac{\alpha - 1}{2}\neq -s_0 +k\frac \gamma 2$ for any $k=1,2,\ldots$.
After repeating the argument finite times, we obtain the estimate \eqref{main_goal}.
\end{proof}

\section{Optimality}
In this section, we prove Theorem \ref{optimality}. We will make use of the following lemma.

\begin{lemma}
\label{ODE_lemma}
For $\varepsilon > 0$, there exists a unique solution $h \in L^\infty((0,1)) \cap C^\infty((0,1])$ of
\begin{equation}
\label{Homogeneous_ODE}
Lh :=h''(r) + \left( \frac{n-2}{r} + \frac{2r}{\varepsilon + r^2} \right) h'(r) - \frac{n-2}{r^2} h(r) = 0, \quad 0 < r <1
\end{equation}
satisfying $h(1) = 1$.
Moreover, $h \in C([0,1])$, $h(0) = 0$, and for $\beta \ge \frac{2\alpha^2 + \alpha(n-1)}{n-3+\alpha}$, there exist positive constants $C(\varepsilon)$ and $C(\beta)$ such that
\begin{equation}
\label{h_bounds}
r  < h(r) < \min\{ C(\varepsilon)r, r^\alpha\}~ \mbox{and}~ h(r) \ge \frac{1}{C(\beta)} r^\beta (\varepsilon + r^2)^{\frac{\alpha - \beta}{2}} \quad \mbox{for}~~0 < r < 1,
\end{equation}
where $\alpha$ is given by \eqref{alpha_n} and $h$ is strictly increasing in $[0,1]$.
\end{lemma}
\begin{proof}
For $0 < a < 1$, let $h_a \in C^2([a,1])$ be the solution of $Lh_a = 0$ in $(a, 1)$ satisfying $h_a(a) = a$ and $h_a(1) = 1$. Since $Lr > 0$ and $Lr^\alpha < 0$ in $(0,1)$, by the maximum principle and the strong maximum principle,
$$
r < h_a(r) < r^\alpha, \quad a < r < 1.
$$
Sending $a \to 0$ along a subsequence, $h_a \to h$ in $C^2_{\text{loc}}((0,1])$ for some $h \in C([0,1]) \cap C^\infty((0,1])$ satisfying $r \le h(r) \le r^\alpha$, $Lh = 0$ in $(0,1)$, and $h(0) = 0$. By the strong maximum principle,
$$
r < h(r) < r^\alpha, \quad 0 < r < 1.
$$
Let $v = r(1 - r^{1/2}/2)$, by a direct computation,
$$
Lv = -\frac{1}{4}\left( n - \frac{1}{2} \right) r^{-\frac{1}{2}} + \frac{1}{\varepsilon}O(r)  \quad \mbox{as}~~r \to 0.
$$
Hence $Lv < 0$ in $(0, r_0(\varepsilon))$, for some small $r_0(\varepsilon)$. By the maximum principle, we have $h \le C(\varepsilon) v \le C(\varepsilon)r$ in $(0,r_0(\varepsilon))$ for some constant $C(\varepsilon)$.

For $\beta \in \bR$, let $U(r) = r^\beta (\varepsilon + r^2)^{\frac{\alpha - \beta}{2}}$. By a direct computation,
\begin{align*}
LU &= r^{\beta-2}(\varepsilon + r^2)^{\frac{\alpha - \beta}{2}}\left\{ (\beta-\alpha)^2 \left( \frac{r^2}{\varepsilon + r^2} \right)^2 \right. \\
&\left. + [(2\beta + n -1)(\alpha - \beta) + 2\beta] \left( \frac{r^2}{\varepsilon + r^2} \right) + (n-2+\beta)(\beta-1) \right\}, \quad 0 < r < 1.
\end{align*}
Consider the second order polynomial
$$
p(x):= (\beta-\alpha)^2 x^2 + [(2\beta + n -1)(\alpha - \beta) + 2\beta]x +(n-2+\beta)(\beta-1), \quad x \in [0,1].
$$
Since $p(1) = 0$, a sufficient condition for $LU \ge 0$ in $(0,1)$ is
$$
p'(x) = 2(\beta - \alpha)^2x + (2\beta + n -1)(\alpha - \beta) + 2\beta \le 2(\beta - \alpha)^2 + (2\beta + n -1)(\alpha - \beta) + 2\beta \le 0.
$$
This is equivalent to $\beta \ge \frac{2\alpha^2 + \alpha(n-1)}{n-3+\alpha}$. Therefore, $U$ is a subsolution of \eqref{Homogeneous_ODE} when $\beta \ge \frac{2\alpha^2 + \alpha(n-1)}{n-3+\alpha}$. Estimate \eqref{h_bounds} follows.

Next we show $h$ is strictly increasing in $(0,1)$. If not, there exists an $r_0 \in (0,1)$ such that $h'(r_0) = 0$ and $h''(r_0) \le 0$. Since $h(r_0) > 0$, we have $Lh(r_0) < 0$, a contradiction.

Finally, we show the uniqueness of $h$. Let $h_2\in L^\infty((0,1)) \cap C^\infty((0,1])$ be a solution of \eqref{Homogeneous_ODE} in $(0,1)$ satisfying $h_2(1) = 1$. Then $w(r) := h_2(r) / h(r)$ satisfies
$$
(G w')' = 0, \quad 0<r<1,
$$
where $G = h^2 r^{n-2}(\varepsilon + r^2)$.
Therefore, for some constants $C_0$ and $C_1$, we have
$$
h_2(r) = h(r) w(r) = h(r) \int_{r}^1 \frac{C_0}{h^2(s) s^{n-2} (\varepsilon + s^2)} \, ds + C_1 h(r), \quad 0 < r <  1.
$$
By the first inequality in \eqref{h_bounds},
$$
h(r) \int_r^1 \frac{1}{h^2(s) s^{n-2} (\varepsilon + s^2)} \, ds \to +\infty
$$
as $r \to 0$. Therefore, since $h_2$ and $h$ are bounded, $C_0= 0$, $C_1 = 1$, and $h_2 = h$.
\end{proof}

\begin{proof}[Proof of Theorem \ref{optimality}]
\textbf{Step 1.} By the elliptic theory, the fact that $\widetilde\Omega$ is symmetric in $x_1$, and the fact that $\varphi$ is odd in $x_1$, we know that $u$ is smooth and $u$ is odd in $x_1$. In $\{ (x',x_n) \in \bR^n ~\big|~ |x'| < 1 \}$, $f$ and $g$ can be written as
$$
f(x') = \frac{|x'|^2}{2} + O(|x'|^{4}) \quad \mbox{and} \quad g(x') = -\frac{|x'|^2}{2}  + O(|x'|^{4}),
$$
respectively. In $\Omega_1$, where $\Omega_r$ is defined as in \eqref{domain_def_Omega}, we define
$$\bar{u}(x') = \fint_{-\frac{\varepsilon}{2} + g(x') < x_n < \frac{\varepsilon}{2} + f(x')} u \, dx_n.$$
Then $\bar{u}$ satisfies
$$\dv((\varepsilon +|x'|^2) \nabla \bar{u}) = \dv F, \quad \mbox{in}\,\,B_1 \subset \bR^{n-1},$$
where
$$
F_i = -2 (x_i + O(|x'|^{3})) \overline{\partial_n u x_n} + O(|x'|^{4})\partial_i \bar{u},
$$
$\overline{\partial_n u x_n}$ is the average of $\partial_n u x_n$ with respect to $x_n$ in $(-\varepsilon/2 + g(x'), \varepsilon/2 + f(x'))$. We have, by \eqref{grad_n_u_bound}, $|x_n| \le C(\varepsilon + |x'|^2)$, and by \eqref{main_goal},
\begin{equation}
\label{F_bound_2}
|F(x')| \le C(n) |x'|(\varepsilon + |x'|^2) \quad x' \in B_1.
\end{equation}
Again, we denote $Y_{k,i}$ to be a $k$-th degree normalized spherical harmonics so that $\{Y_{k,i}\}_{k,i}$ forms an orthonormal basis of $L^2(\bS^{n-2})$, $Y_{1,1}$ to be the one after normalizing $\left.x_1\right|_{\bS^{n-2}}$, and $x' = (r ,\xi)$. Since $\bar u$ is odd with respect to $x_1 = 0$, and in particular $\bar{u}(0) = 0$, we have the following decomposition
\begin{equation}
\label{u_bar_expansion}
\bar{u}(x') = U_{1,1}(r)Y_{1,1}(\xi) + \sum_{k=2}^\infty \sum_{i=1}^{N(k)} U_{k,i}(r)Y_{k,i}(\xi), \quad x' \in B_1\setminus\{0\},
\end{equation}
where $U_{k,i}(r) = \int_{\bS^{n-2}} \bar{u}(r,\xi) Y_{k,i}(\xi) \, d\xi$ and $U_{k,i} \in C([0,1)) \cap C^\infty((0,1))$. Since $\bar{u}(0) = 0$ and $\varepsilon + |x'|^2$ is independent of $\xi$, $U_{1,1}$ satisfies $U_{1,1}(0) = 0$ and
\begin{equation*}
\label{U_11_equation}
LU_{1,1}:= U_{1,1}''(r) + \left( \frac{n-2}{r} + \frac{2r}{\varepsilon + r^2} \right) U_{1,1}'(r) - \frac{n-2}{r^2} U_{1,1}(r) = H(r), \quad 0 < r <1,
\end{equation*}
where
\begin{align*}
H(r) &= \int_{\bS^{n-2}} \frac{(\dv F) Y_{1,1}(\xi)}{\varepsilon+r^2} \, d\xi = \int_{\bS^{n-2}} \frac{\partial_r F_r + \frac{1}{r} \nabla_\xi F_\xi}{\varepsilon + r^2} Y_{1,1}(\xi) \, d\xi\\
&= \partial_r \left(\int_{\bS^{n-2}} \frac{F_r}{\varepsilon + r^2} Y_{1,1}(\xi) \, d\xi \right) + \int_{\bS^{n-2}} \frac{2rF_r Y_{1,1}}{(\varepsilon+r^2)^2} - \frac{F_\xi \nabla_\xi Y_{1,1}}{r(\varepsilon + r^2)}\, d\xi\\
&=: A'(r) + B(r), \quad 0 < r < 1,
\end{align*}
and $A(r),B(r) \in C^1([0,1))$ satisfy, in view of \eqref{F_bound_2}, that
\begin{equation}
\label{AB_bounds}
|A(r)| \le C(n)r, \quad |B(r)| \le C(n), \quad 0 < r < 1.
\end{equation}

\textbf{Step 2.}
We will prove, for some constant $C_1(\varepsilon)$, that
\begin{equation}
\label{U_11_formula}
U_{1,1}(r) = C_1 (\varepsilon) h(r) + O(r^{1+\alpha}), \quad 0<r<1.
\end{equation}
We use the method of reduction of order to write down a bounded solution $v$ satisfying $Lv = H$ in $(0,1)$, and then give an estimate $v = O(r^{1+ \alpha})$.

Let $h \in C([0,1]) \cap C^\infty((0,1])$ be the solution of $Lh = 0$ satisfying $h(0) = 0$ and $h(1) = 1$ as in Lemma \ref{ODE_lemma}. Let $v = h w$ and
$$
w(r) := \int_0^r\frac{1}{h^2(s) s^{n-2} (\varepsilon + s^2)} \int_0^s h(\tau) \tau^{n-2} (\varepsilon + \tau^2) H(\tau) \, d\tau ds, \quad 0 < r < 1.
$$
By a direct computation,
$$
Lv = L(hw) = h w'' + \left[ 2h' + \left( \frac{n-2}{r} + \frac{2r}{\varepsilon + r^2} \right)h \right]w' = \frac{h}{G}(Gw')' = H,
$$
where $G= h^2 r^{n-2}(\varepsilon + r^2)$. By \eqref{AB_bounds} and the fact that $h' > 0$, we can estimate
\begin{align*}
&\int_0^s h(\tau) \tau^{n-2} (\varepsilon + \tau^2) H(\tau) \, d\tau \\
& = \int_0^s h(\tau) \tau^{n-2} (\varepsilon + \tau^2) A'(\tau) \, d\tau + O(1)h(s) s^{n-1}(\varepsilon + s^2)\\
&= -\int_0^s h' \tau^{n-2} (\varepsilon + \tau^2) A(\tau) \, d\tau + O(1)h(s) s^{n-1}(\varepsilon + s^2)\\
&= O(1) s^{n-1}(\varepsilon + s^2) \int_{0}^s h'(\tau) \, d\tau+ O(1)h(s) s^{n-1}(\varepsilon + s^2)\\
&= O(1)h(s) s^{n-1}(\varepsilon + s^2).
\end{align*}
Therefore, using \eqref{h_bounds},
$$
|v(r)| \le Ch(r) \int_0^r \frac{s}{h(s)} = O(r^{1+\alpha}), \quad 0 < r < 1.
$$
Since $U_{1,1} - v$ is bounded and satisfies $L(U_{1,1} - v) = 0$ in $(0,1)$, by Lemma \ref{ODE_lemma}, $U_{1,1} - v = C_1(\varepsilon)h$. Hence \eqref{U_11_formula} follows.

\textbf{Step 3.} We will show that $C_1 (\varepsilon) > \frac{1}{C}$ for some positive $\varepsilon$-independent constant $C$.

Denote $x = (r ,\xi , x_n) \in \bR_+ \times \bS^{n-2} \times \bR$ and write \eqref{equzero} under the condition of Theorem \ref{optimality} as the following:
\begin{equation}
\label{equzero_cylindrical}
\left\{
\begin{aligned}
u_{rr} + \frac{n-2}{r}u_r + \frac{1}{r^2}\Delta_{\bS^{n-2}} u + u_{nn} &= 0 \quad\mbox{in}~~B_5 \setminus \overline{(D_1 \cup D_2)}\\
\frac{\partial u}{\partial \nu} &= 0 \quad \mbox{on}~\partial{ D}_{i},~i=1,2,\\
 u &= x_1 \quad \mbox{on } \partial  B_5.
\end{aligned}
\right.
\end{equation}
Let
$$
\hat{u}(r,x_n) = \int_{\bS^{n-2}} u(r,\xi,x_n)Y_{1,1}(\xi) \, d \xi.
$$
Since $u$ is odd in $x_1$, we have $\hat u(0 ,x_n) = 0$ for any $x_n$. Multiplying \eqref{equzero_cylindrical} by $Y_{1,1}(\xi)$ and integrating over $\bS^{n-2}$, we know that $\hat u(r,x_n)$ satisfies
\begin{equation}
\label{equation_u_hat}
\left\{
\begin{aligned}
\hat{u}_{rr} + \frac{n-2}{r}\hat{u}_r - \frac{n-2}{r^2}\hat{u} + \hat{u}_{nn} &= 0 \quad\mbox{in}~~\widehat{B}_5 \setminus \overline{(\widehat{D}_1 \cup \widehat{D}_2)}\\
\frac{\partial \hat u}{\partial \nu} &= 0 \quad \mbox{on}~\partial{\widehat D}_{i},~i=1,2,\\
\hat u &= 0 \quad \mbox{on } \{r = 0\},\\
\hat u &= r \quad \mbox{on } \partial \widehat B_5,
\end{aligned}
\right.
\end{equation}
where
\begin{align*}
\widehat B_5 &:=\{(r,x_n)\in \bR_+\times \bR ~\big|~ r^2 + x_n^2 < 25 \},\\
\widehat D_i &:= \{(r,x_n)\in \bR_+\times \bR ~\big|~ r^2 + (x_n - (-1)^i (1 + \varepsilon/2))^2 < 1\},
\end{align*}
and $\nu$ is the unit inner normal of $\partial \widehat{D}_i$. Clearly $\hat v(r) = r$ satisfies the first line of \eqref{equation_u_hat}, and $\frac{\partial \hat v}{\partial \nu} < 0$ on $\partial{\widehat D}_{i}$, $i = 1,2$. Thus, we know that $r$ is a subsolution of \eqref{equation_u_hat}, and therefore $\hat{u} \ge r$. Then
$$U_{1,1}(r)= \fint_{-\varepsilon/2 + g(x') <x_n<\varepsilon/2 + f(x')} \hat{u}(r,x_n) \, dx_n \ge r.$$
By \eqref{U_11_formula}, \eqref{h_bounds}, and the above, we have
$$
r \le U_{1,1}(r) = C_1(\varepsilon)h(r) + O(r^{1+\alpha}) \le C_1(\varepsilon) r^\alpha + \frac{1}{2}r, \quad\forall 0 < r \le r_0,
$$
where $r_0$ is a small constant independent of $\varepsilon$,
which implies that
$$
C_1(\varepsilon) \ge  \frac{1}{2}r_0^{1-\alpha}.
$$

\textbf{Step 4.} Completion of the proof of Theorem \ref{optimality}.

It follows, in view of \eqref{U_11_formula}, Step 3,  and \eqref{h_bounds}, that there exists some positive constant $r_0$ independent of $\varepsilon$ such that
\begin{equation}
\label{U_11_lowerbound}
U_{1,1}(r) \ge \frac{1}{C} h(r) + O(r^{1+ \alpha}) \ge \frac{1}{2C} h(r), \quad 0 < r \le r_0.
\end{equation}
By \eqref{h_bounds},
\begin{equation}
\label{h1_lowerbound}
h(r) \ge \frac{1}{C} r^\beta (\varepsilon + r^2)^{\frac{\alpha - \beta}{2}}\quad  \mbox{for}~~\beta = \frac{2\alpha^2 + \alpha(n-1)}{n-3+\alpha}.
\end{equation}
By \eqref{u_bar_expansion}, \eqref{U_11_lowerbound}, and \eqref{h1_lowerbound}, we have
$$
\left( \int_{\bS^{n-2}} |\bar{u}(\sqrt{\varepsilon},\xi)|^2 \, d\xi \right)^{1/2} \ge |U_{1,1}(\sqrt{\varepsilon})| \ge \frac{1}{C}h(\sqrt{\varepsilon}) \ge \frac{1}{C} \varepsilon^{\alpha/2}.
$$
Then, there exists a $\xi_0 \in \bS^{n-2}$ such that $|\bar{u}(\sqrt{\varepsilon}, \xi_0)| \ge \frac{1}{C} \varepsilon^{\alpha/2}$. Since $\bar u$ is the average of $u$ in the $x_n$ direction, there exists an $x_n$ such that
\begin{equation}
\label{u_lower_bound}
|u(\sqrt{\varepsilon}, \xi_0, x_n)| \ge \frac{1}{C} \varepsilon^{\alpha/2}.
\end{equation}
Estimate \eqref{grad_u_lower_bound} follows from \eqref{u_lower_bound} and $u(0) = 0$.
\end{proof}

\bibliographystyle{amsplain}

\end{document}